\documentclass{amsart}

\usepackage{amsmath,amsthm,amssymb,url}

\newtheorem{theorem}{Theorem}[section]
\newtheorem{lemma}[theorem]{Lemma}

\newtheorem{corollary}[theorem]{Corollary}

\newcommand{\BB}{\mathbb{B}}
\newcommand{\HH}{\mathbb{H}}
\newcommand{\NN}{\mathbb{N}}

\newcommand{\ZZ}{\mathbb{Z}}
\newcommand{\RR}{\mathbb{R}}
\newcommand{\CC}{\mathbb{C}}

\newcommand{\lpb}{{\ell^p(\BB)}}

\newcommand{\st}{\; | \;}

\newcommand{\mdl}[1]{{\mathcal #1}}

\newcommand{\cE}{{\mathcal E}}

\title[Oscillation and the mean ergodic theorem]{Oscillation and the mean ergodic theorem\\ for
uniformly convex Banach spaces}

\author{Jeremy Avigad and Jason Rute}


\thanks{Avigad's work has been partially supported by NSF grant DMS-1068829 and AFOSR grant FA9550-12-1-0370, and Rute's work has been partially supported by NSF grants DMS-1068829 and DMS-0901020. We are grateful to Laurentiu Leau\c{s}tean for helpful comments and corrections to the proof in Section~\ref{old:main:thm}, to Ulrich Kohlenbach for corrections, and to C{\'e}dric Arhancet for pointing us to the use of UMD spaces in Theorem~\ref{partial:converse:thm}.}

\begin{document}

\begin{abstract}
Let $\BB$ be a $p$-uniformly convex Banach space, with $p \geq 2$. Let $T$ be a linear operator on $\BB$, and let $A_n x$ denote the ergodic average $\frac{1}{n} \sum_{i< n} T^n x$. We prove the following variational inequality in the case where $T$ is power bounded from above and below: for any increasing sequence  $(t_k)_{k \in \NN}$ of natural numbers we have $\sum_k \| A_{t_{k+1}} x - A_{t_k} x \|^p \leq C \| x \|^p$, where the constant $C$ depends only on $p$ and the modulus of uniform convexity. For $T$ a nonexpansive operator, we obtain a weaker bound on the number of $\varepsilon$-fluctuations in the sequence. We clarify the relationship between bounds on the number of $\varepsilon$-fluctuations in a sequence and bounds on the rate of metastability, and provide lower bounds on the rate of metastability that show that our main result is sharp.
\end{abstract}

\maketitle

\section{Introduction}
\label{introduction:section}

A Banach space $\BB$ is said to be \emph{uniformly convex} if for every $\varepsilon \in (0,2]$ there exists a $\delta \in (0,1]$ such that for all $x, y \in \BB$, if $\| x \| \leq 1$, $\| y \| \leq 1$, and $\| x - y \| \geq \varepsilon$, then $\| (x + y)/2 \| \leq 1 - \delta$. Any function $\eta(\varepsilon)$ returning such a $\delta$ for each $\varepsilon$ is called a \emph{modulus of uniform convexity}. 
A Banach space is said to be \emph{$p$-uniformly convex} if, for some $C > 0$, the function $\eta(\varepsilon) = C \varepsilon^p$ is a modulus of uniform convexity.  
Pisier \cite{pisier:75,pisier:11} has shown that every uniformly convex Banach space is isomorphic to a $p$-uniformly convex Banach space for some $p \geq 2$. In particular, for $p \geq 2$ and any measure space $X$, $L^p(X)$ is $p$-uniformly convex. 

Let $\BB$ be a uniformly convex Banach space, and let $T$ be a nonexpansive linear operator, that is, a linear map satisfying $\| T x \| \leq \| x \|$ for every $x \in \BB$. For each $n \geq 1$, let $A_n x$ denote the ergodic average $\frac{1}{n} \sum_{i < n} T^n x$. A version of the mean ergodic theorem due to Garrett Birkhoff \cite{birkhoff:39} implies that the sequence $(A_n x)$ converges. (See Krengel \cite[Chapter 2]{krengel:85} for variations and strengthenings.) In the special case where $\BB$ is a Hilbert space, Jones, Ostrovskii, and Rosenblatt \cite{jones:et:al:96} prove the following ``square function'' variational inequality:
\begin{theorem}
\label{jro:thm}
Let $\HH$ be a Hilbert space and let $T$ be any nonexpansive linear operator. Then for any $x$ in $\HH$ and any increasing sequence $(t_k)_{k \in \NN}$, 
\[
 \sum_k \|A_{t_{k+1}} x - A_{t_k} x\|^2 \leq 625 \| x \|^2.
\]
\end{theorem}
If $(a_n)$ is any finite or infinite sequence of elements of $\BB$ and $\varepsilon > 0$, we will say that $(a_n)$ \emph{admits $k$ $\varepsilon$-fluctuations} if there are $i_1 \leq j_1 \leq \ldots \leq i_k \leq j_k$ such that, for each $u=1,\ldots,k$, $\| a_{j_u} - a_{i_u}\| \geq \varepsilon$. Theorem~\ref{jro:thm} implies that for every $\varepsilon > 0$, the sequence of ergodic averages $(A_n x)$ admits at most $625 (\| x \| / \varepsilon)^2$-many $\varepsilon$-fluctuations, and hence the sequence converges.

Theorem~\ref{jro:thm} can therefore be viewed as a strong quantitative version of the mean ergodic theorem for a Hilbert space. The methods of Jones, Ostrovskii, and Rosenblatt \cite{jones:et:al:96}, based on Fourier analysis, form the starting point for a number of pointwise variational results for sequences of ergodic averages in the spaces $L^p(X)$ by Jones, Kaufman, Rosenblatt, and Wierdl \cite{jones:et:al:98}. These pointwise results imply the following variational inequality regarding convergence of ergodic averages in the $L^p$ norm (see \cite[Theorem B]{jones:et:al:98}):
\begin{theorem}
\label{thm:jro:mean}
Let $p \geq 2$ and let $T$ be the isometry on $L^p(X)$ arising from a measure-preserving tranformation of the finite measure space $X$. Then for any $f$ in $L^p(X)$ and any increasing sequence $(t_k)_{k \in \NN}$,
\[
 \sum_k \|A_{t_{k+1}} f - A_{t_k} f\|_p^p \leq C \| f \|^p_p 
\]
for some constant $C$ that depends only on $p$.
\end{theorem}
\noindent In fact, their pointwise results are considerably stronger; see the discussion at the end of Section~\ref{lemmas:section}.

Our main result generalizes Theorem~\ref{thm:jro:mean} to an arbitrary uniformly convex Banach space. It applies not only when $T$ is an isometry, but more generally when $T$ is power bounded from above and below, as in the hypothesis of the next theorem.
\begin{theorem}
\label{new:main:thm}
Let $p \geq 2$ and let $\BB$ be any $p$-uniformly convex Banach space. Let $T$ be a linear operator on $\BB$ satisfying $B_1 \| y \| \leq \| T^n y \| \leq B_2 \| y \|$ for every $n$ and $y \in \BB$, for some $B_1, B_2 > 0$. Then for any $x$ in $\BB$ and any increasing sequence $(t_k)_{k \in \NN}$,
\[
 \sum_k \|A_{t_{k+1}} x - A_{t_k} x\|^p \leq C \| x \|^p
\]
for some constant $C$. If $\eta(\varepsilon)= K \varepsilon^p$ is a modulus of uniform convexity for $\BB$, the constant $C$ depends only on $B_1$, $B_2$, $K$, and $p$.
\end{theorem}
The difficulty in proving Theorem~\ref{new:main:thm} is that the Fourier-analytic methods used by Jones, Ostrovskii, and Rosenblatt \cite{jones:et:al:96} are not available in this general setting. But Jones, Rosenblatt, and Wierdl \cite{jones:et:al:03} later developed methods of proving pointwise variational results for $L^p$ spaces that bypass the need for Fourier analysis. Our strategy is to adapt their methods to \emph{Banach-valued} $L^p$ spaces $L^p(X {;\,} \BB)$, and then transfer the result back to the original Banach space, $\BB$. As a side effect, we also obtain pointwise results for the spaces $L^p(X {;\,} \BB)$.

The proof of Theorem~\ref{jro:thm} in \cite{jones:et:al:96} first establishes the result in the case that $T$ is an isometry, and then invokes Sz.-Nagy's dilation theorem \cite{nagy:et:al:10}, which says that any Hilbert space with a nonexpansive map can be embedded into a larger Hilbert space with an isometry that projects to the original nonexpansive map. This strategy is not available for arbitrary uniformly convex Banach spaces, and the assumption that $T$ is power bounded from below seems to be essential to our proof of Theorem~\ref{new:main:thm}. Theorem~\ref{new:main:thm} implies that for every $\varepsilon > 0$, the number of $\varepsilon$-fluctuations in a sequence $(A_n x)$ in a $p$-uniformly convex Banach space is $O(\rho^p)$, where $\rho = \| x \| / \varepsilon$. In the case where $T$ is a nonexpansive map, we obtain the following weaker bound.
\begin{theorem}
\label{old:main:thm}
Let $p \geq 2$ and let $\BB$ be any $p$-uniformly convex Banach space, and let $T$ be a nonexpansive linear operator on $\BB$. Then for any $x$ in $\BB$ and any $\varepsilon > 0$, there are at most $C \rho^{p + 1} \log \rho$-many $\varepsilon$ fluctuations in $(A_n x)$, for a constant $C$ that depends only on $p$ and $K$, where 
$\eta(\varepsilon)= K \varepsilon^p$ is a modulus of uniform convexity for $\BB$.
\end{theorem}
\noindent Here we use an entirely different argument, drawing on calculations by Kohlenbach and Leu\c{s}tean \cite{kohlenbach:leustean:09}, which, in turn, draw on those in Birkhoff's proof \cite{birkhoff:39}.

The outline of this paper is as follows. Theorem~\ref{new:main:thm} is proved in Section~\ref{new:main:section}, modulo two central lemmas adapted from \cite{jones:et:al:03}, which are proved in Section~\ref{lemmas:section}. In Section~\ref{old:main:section}, we prove Theorem~\ref{old:main:thm}. In Section~\ref{quantitative:section}, we consider various quantitative data that can be associated to a convergence theorem, and clarify the relationship between the results here and results having to do with metastability \cite{avigad:et:al:10,kohlenbach:leustean:09,tao:08,kohlenbach:unp:e,kohlenbach:leustean:unp,kohlenbach:schade:12,walsh:12}. In Section~\ref{lower:bounds:section}, we provide lower bounds on the rate of metastability that show that the result of Theorem~\ref{new:main:section} is sharp.

\section{The variational inequality}
\label{new:main:section}

Jones, Rosenblatt, and Weirdl \cite{jones:et:al:03} prove a variety of variational inequalities for sequences of generalized ergodic averages in the classical $L^p$ spaces. Their starting point is to prove related inequalities for the space $\ell^2$, consisting of functions $f : \ZZ \to \RR$ with norm $\| f \|_{\ell^2} = \left( \sum_x f^2(x) \right)^{1/2}$, and then to transfer these results to the $L^p$ spaces using the Calder\'on transfer principle, as well as Calder\'on-Zygmund decomposition.

If $X$ is any measure space and $\BB$ is any Banach space, the classical spaces $L^p(X)$ can be generalized to spaces $L^p(X {;\,} \BB)$, consisting of measurable functions $f:X \to \BB$ for which the norm
\[
 \| f \|_{L^p(X {;\,} \BB)} = \left(\int_X \|f\|_\BB^p\right)^{1/p}
\] is finite. Precise definitions can be found in Pisier~\cite{pisier:75,pisier:11}. When $X$ is the set of integers with counting measure, $L^p(X {;\,} \BB)$ is the space that we will denote $\lpb$. 

We will prove Theorem~\ref{new:main:thm} in two steps:
\begin{enumerate}
 \item Generalize the $\ell^2$ results of \cite{jones:et:al:03} to $\lpb$, where $\BB$ is a $p$-uniformly convex Banach space.
 \item Use a novel transfer argument to transfer the result back to $\BB$, rather than to $L^p(X {;\,} \BB)$.  
\end{enumerate}
The key insight is that one can use martingale inequalities for $\lpb$ in the first step in place of orthogonal decomposition in $\ell^2$. Beyond that, the changes that need to be made to the arguments of \cite{jones:et:al:03} are minor, such as replacing appeals to Cauchy's inequality by appeals to H\"older's inequality. For completeness, we will spell out these modifications explicitly in the next section, for the particular case needed for the proof of Theorem~\ref{new:main:thm}. The rest of this section is devoted to the proof of Theorem~\ref{new:main:thm}, modulo two lemmas proved in Section~\ref{lemmas:section}.

Let $A_0 \subseteq A_1 \subseteq A_2 \subseteq \ldots$ be a filtration of $X$, that is, a sequence of $\sigma$-subalgebras of $X$. A sequence $(M_n)_{n \geq 0}$ of elements of $L^p(X {;\,} \BB)$ is said to be a \emph{martingale} for this filtration if each $M_n$ is $A_n$-measurable, and for any $n \leq k$, $M_n = \mathbb{E}(M_k | A_n)$. We will use the forward direction of the following elegant equivalence, due to Pisier~\cite{pisier:75,pisier:11}:
\begin{theorem}
\label{pisier:thm}
Let $p$ satisfy $2 \leq p < \infty$, and let $\BB$ be a Banach space. The following are equivalent:
\begin{enumerate}
 \item $\BB$ is isomorphic to a $p$-uniformly convex Banach space.
 \item There is a constant, $C$, such that if $(M_n)_{n \geq 0}$ is any martingale in $L^p(X {;\,} \BB)$, then
\[
\| M_0 \|^p_{L^p(X {;\,} \BB)} + \sum_{n \geq 0} \| M_{n+1} - M_n \|^p_{L^p(X {;\,} \BB)} \leq C \sup_{n \geq 0} 
\| M_n \|^p_{L^p(X {;\,} \BB)}.
\]
\end{enumerate}
If $\BB$ has modulus of uniform convexity $\eta(\varepsilon) = K \varepsilon^p$, the constant $C$ in (2) depends only on $p$ and $K$. 
\end{theorem}

We now turn our attention to $\lpb$, where $\BB$ is a $p$-uniformly convex Banach space. Following the notation used in \cite{jones:et:al:03}, let $\sigma_n$ denote the dyadic $\sigma$-algebra of subsets of $\ZZ$ generated by the intervals of the form $[h \cdot 2^n, (h + 1) \cdot 2^n)$. For any $f$ in $\lpb$ let $\cE_n f$ denote the expectation with respect to $\sigma_n$. Notice that the sequence $(\cE_n f)_{n \geq 0}$ is a \emph{reverse} martingale, since the $\sigma_n$'s become coarser as $n$ increases. As a corollary of Theorem~\ref{pisier:thm}, we have:
\begin{corollary}
\label{pisier:cor}
  Let $\BB$ be $p$-uniformly convex, for $p \geq 2$.  Then there is a constant $C$ as in Theorem~\ref{pisier:thm} such that for any increasing sequence $(n_k)_{k \in \NN}$ of natural numbers,
\[
 \sum_{k \geq 0} \| \cE_{n_{k+1}} f - \cE_{n_k} f \|^p_\lpb \leq C \| f \|^p_\lpb.
\]
\end{corollary}

\begin{proof}
It suffices to show $\sum_{k \in [0,m]} \| \cE_{n_{k+1}} f - \cE_{n_k} f \|^p_\lpb \leq C \| f \|^p_\lpb$ for every $m$. But for each fixed $m$ the sequence $\cE_{n_{m+1}} f, \cE_{n_m} f, \ldots, \cE_{n_1} f, \cE_{n_0} f$ is a martingale, and Theorem~\ref{pisier:thm} applies since $\sup_{k \in [0,m+1]} \| \cE_{n_k} \|_\lpb \leq \| f \|_\lpb$.
\end{proof}

For the rest of this section and the next, we let $p \geq 2$ and let $\BB$ denote a $p$-uniformly convex Banach space. Let $T$ be the shift map $T (i) = i + 1$ on $\ZZ$. Then $T$ gives rise to an isometry on $\lpb$, and we can consider the sequence $(A_n f)_{n \geq 1}$ of ergodic averages in $\lpb$ for each $f \in \lpb$. Notice that $(A_n f)$ is explicitly given by $(A_n f)(x) = \frac{1}{n} \sum_{i < n} f(x + i)$ for every $x$. We will prove the following:

\begin{theorem}
\label{lpb:main:thm}
 Let $p \geq 2$, and let $\BB$ be a Banach space with modulus of uniform convexity $\eta(\varepsilon) = K \varepsilon^p$. Then there is a constant $C$ depending only on $p$ and $K$ such that for any increasing sequence $(t_k)_{k \in \NN}$ of positive natural numbers,  
\[
\sum_k \| A_{t_{k+1}} f - A_{t_k} f \|_\lpb^p \leq C \cdot \| f \|^p_\lpb 
\]
\end{theorem}
With $p = 2$ and $\BB = \RR$, this is a special case of \cite[Theorem 1.7]{jones:et:al:03}. In the next section, we will prove the following two lemmas:

\begin{lemma}
\label{lemma:a}
  Let $(t_k)_{k \in \NN}$ have the property that, for each $k$, $t_k \in [2^{k-1}, 2^k)$. Then there is a constant $C$ as in Theorem~\ref{lpb:main:thm} such that
\[
\sum_k \| A_{t_k} f - \cE_k f \|_\lpb^p \leq C \cdot \| f \|^p_\lpb 
\]
\end{lemma}

\begin{lemma}
\label{lemma:b}
  Let $(t_i)_{i \in \NN}$ be any increasing sequence of natural numbers. Then there is a constant $C$ as in Theorem~\ref{lpb:main:thm} such that
\[
 \sum_k \sum_{i \in S_k} \| A_{t_{i+1}} f - A_{t_i} f \|_\lpb^p \leq C \cdot \| f \|^p_\lpb 
\]
  where $S_k = \{ i \; | \; t_i, t_{i+1} \in [2^{k-1}, 2^k]\}$. 
\end{lemma}

Once again, when $p = 2$ and $\BB = \RR$, Lemmas~\ref{lemma:a} and \ref{lemma:b} are special cases of Theorems $A'$ and $B$ of \cite{jones:et:al:03}. The proof of Theorem~\ref{lpb:main:thm} follows easily from these two lemmas:
\begin{proof}[Proof of Theorem~\ref{lpb:main:thm}]
  Given the sequence $(t_k)$, let $S = \bigcup_k S_k$ (the ``short'' increments), where $S_k$ is as defined in Lemma~\ref{lemma:b}. Let $L = \NN - S$ (the ``long'' increments). Also let the sequence $(n_k)$ be such that $t_k \in [2^{{n_k}-1}, 2^{n_k})$ for each $k$. Then we have
\begin{align*}
\sum_k \| A_{t_{k+1}} f - A_{t_k} f \|_\lpb^p & = \sum_{k \in L} \| A_{t_{k+1}} f - A_{t_k} f \|_\lpb^p +
  \sum_{i \in S} \| A_{t_{i+1}} f - A_{t_i} f \|_\lpb^p \\
& \leq \sum_{k \in L} \| A_{t_{k+1}} f - \cE_{n_{k+1}} f \|_\lpb^p
 + \sum_{k \in L} \| A_{t_k} f - \cE_{n_k} f \|_\lpb^p \\
& \quad \quad + \sum_{k \in L} \| \cE_{n_{k+1}} f - \cE_{n_k} f \|_\lpb^p \\
& \quad \quad + \sum_k \sum_{i \in S_k} \| A_{t_{i+1}} f - A_{t_i} f \|_\lpb^p.
\end{align*}
Lemma~\ref{lemma:a} provides the requisite bounds on the first two terms, Corollary~\ref{pisier:cor} provides the requisite bound on the third, and Lemma~\ref{lemma:b} provides the requisite bound on the fourth.
\end{proof}

From Theorem~\ref{lpb:main:thm}, we obtain our main theorem.

\begin{proof}[Proof of Theorem~\ref{new:main:thm}] 
 Fix a $p$-uniformly convex Banach space $\BB$, $B_1, B_2 > 0$, $x \in \BB$, and a linear map $T : \BB \to \BB$ satisfying $B_1 \| y \| \leq \| T^n y \| \leq B_2 \| y \|$ for every $n \in \NN$ and $y \in \BB$. We need to show
\[
 \sum_k \|A_{t_{k+1}} x - A_{t_k} x\|_\BB^p \leq C \| x \|_\BB^p
\]
 for an appropriate constant $C$. Note that it suffices to show
\[
 \sum_{k < m} \|A_{t_{k+1}} x - A_{t_k} x\|_\BB^p \leq C \| x \|_\BB^p
\]
 for any $m$. We shift the setting to $\lpb$ by choosing an $N$ much larger than $t_{m + 1}$ and defining
\[
 f(i) = \left\{ 
  \begin{array}{ll}
     T^i x & \mbox{if $i \in [0, N)$} \\
     0 & \mbox{otherwise.}
  \end{array}
 \right.
\]
We have $\| f \|_{\lpb}^p \leq N \cdot B_2^p \cdot \| x \|^p$. Also, we have $(A_n f)(i) = A_n(T^i x)$ provided $i + n < N$ and $i \geq 0$, where $(A_n f)$ denotes the ergodic average with respect to the shift map on $\lpb$, and $A_n(T^i x)$ denotes the ergodic average with respect to the map $T$ on $\BB$. Thus we have
\begin{align*}
\sum_{k < m} \| A_{t_{k+1}} f - A_{t_k} f \|^p_\lpb & = \sum_{k < m} \sum_{i\in \ZZ} \| A_{t_{k+1}} f(i) - A_{t_k} f(i) \|_\BB^p \\
& = \sum_{k < m} \sum_{i \in [0, N)} \| A_{t_{k+1}} (T^i x) - A_{t_k} (T^i x) \|_\BB^p +O(1),
\end{align*}
where the term $O(1)$ accounts for averages which overlap the boundary of $[0,N)$. The constant implicit in that term depends on $t_m$ but not $N$. Turning the equation around and applying Theorem~\ref{lpb:main:thm}, we have
\begin{align*}
\sum_{k < m} \sum_{i \in [0, N)} \| A_{t_{k+1}} (T^i x) - A_{t_k} (T^i x) \|^p_\BB &
  \leq C \cdot \| f \|^p_\lpb + O(1) \\
& \leq C \cdot N \cdot B_2^p \cdot \| x \|^p + O(1).
\end{align*}
On the other hand, we have
\begin{align*}
\sum_{k < m} \sum_{i \in [0, N)} \| A_{t_{k+1}} (T^i x) - A_{t_k} (T^i x) \|^p_\BB &
  =  \sum_{k < m} \sum_{i \in [0, N)} \| T^i (A_{t_{k+1}} x - A_{t_k} x) \|^p_\BB \\
& \geq B_1^p \cdot N \cdot \sum_{k < m} \| A_{t_{k+1}} x - A_{t_k} x \|^p_\BB.
\end{align*}
Combining these inequalities and dividing by $B_1^p N$ yields
\[
\sum_{k < m} \| A_{t_{k+1}} x - A_{t_k} x \|^p_\BB \leq C \cdot (B_2 / B_1)^p \cdot \| x \|^p + O(1 / N).
\]
Letting $N$ approach infinity yields the desired conclusion.
\end{proof}

\section{The main lemmas}
\label{lemmas:section}

We now turn to Lemma~\ref{lemma:a} and \ref{lemma:b}. We will prove them by adapting the proofs of Theorems $A'$ and $B$ in \cite{jones:et:al:03} from the setting of $\ell^2$ to $\lpb$, using Corollary~\ref{pisier:cor} in place of orthogonal decomposition in $\ell^2$, and replacing standard inequalities in $\ell^2$ with their counterparts in $\ell^p$. The changes are fairly straightforward, but we spell out the details for completeness.

\newtheorem*{lemma:a}{Lemma \ref{lemma:a}}
\begin{lemma:a}
  Let $(t_k)_{k \in \NN}$ have the property that, for each $k$, $t_k \in [2^{k-1}, 2^k)$. Then there is a constant $C$ as in Theorem~\ref{lpb:main:thm} such that
\[
\sum_k \| A_{t_k} f - \cE_k f \|_\lpb^p \leq C \cdot \| f \|^p_\lpb 
\]
\end{lemma:a}

\begin{proof}
Let $(t_k)$ be as in the hypothesis, and let $d_1 f = f - \cE_1 f$, $d_n f = \cE_{n-1} f - \cE_n f$ for each $n > 1$, and $d_n f = 0$ for each $n \leq 0$. We will show
\begin{equation*}
\| A_{t_k} d_n - \cE_k d_n \|^p_\lpb \leq c \cdot 2^{-|n-k|} \| d_n \|^p_\lpb.
\end{equation*}
This is sufficient, because then we have
\begin{align*}
  \sum_k \| A_{t_k} f - \cE_k f \|^p_\lpb & \leq \sum_k \Big(\sum_{n \in \ZZ} \| A_{t_k} d_n - \cE_k d_n \|_\lpb\Big)^p \\
  & \leq \sum_k \Big(\sum_{n \in \ZZ} (c \cdot 2^{-|n-k|})^{1/p} \cdot \| d_n \|_\lpb\Big)^p
\end{align*}
and, using Young's inequality in the form $\| \alpha * \beta \|^p_{\ell^p} \leq \| \alpha \|^p_{\ell^1} \| \beta \|^p_{\ell^p}$ with $\alpha(n) =  c \cdot 2^{-|n|}$ and $\beta(n) = \| d_n \|_\lpb$, we can continue
\begin{align*}
\ldots  & \leq \Big(\sum_{n \in \ZZ} (c \cdot 2^{-|n|})^{1/p}\Big)^p \sum_{n \in \ZZ} \| d_n \|^p_\lpb \\
  & \leq c' \cdot \| f \|^p_\lpb,
\end{align*}
using Corollary~\ref{pisier:cor} in the last step.

First, consider the case where $n > k$. This implies $\cE_k d_n = d_n$, so  
\begin{align*}
  \| A_{t_k} d_n - \cE_k d_n \|^p_\lpb & = \| A_{t_k} d_n - d_n \|^p_\lpb \\
& = \sum_{x \in \ZZ} \| (A_{t_k} d_n)(x) - d_n(x) \|^p_\BB \\
& = \sum_{h \in \ZZ} \sum_{x \in D_h} \| (A_{t_k} d_n)(x) - d_n(x) \|^p_\BB,
\end{align*}
where $D_h = [h \cdot 2^{n-1}, (h + 1) \cdot 2^{n-1})$. Notice that $d_n(x)$ is constant on each interval $D_h$, and $(A_{t_k} d_n)(x)$ will have the same value as $d_n(x)$for $x \in D_h$, unless $x$ is close enough to the end of the interval so that $[x, x + t_k)$ extends into $D_{h+1}$. For each $h$, let $m_h$ denote the constant value of $\|d_n(x)\|_\BB$ on the interval $D_h$. Then for any $h$, $\| (A_{t_k} d_n)(x) - d_n(x) \|_\BB$ has the value $0$ except for at most $t_k$ values of $x$, and has value less than or equal to $2 \max (m_h, m_{h+1})$ for these exceptional values of $x$. Thus we have
\begin{align*}
\ldots & \leq t_k \cdot 2^p \cdot \sum_{h \in \ZZ} \max(m_h, m_{h+1})^p \\
& \leq 2^k \cdot 2^p \cdot \Big( \sum_{h \in \ZZ} m_h^p + \sum_{h \in \ZZ} m_{h+1}^p \Big) \\
& = 2 \cdot 2 \cdot  2^p \cdot 2^{k - n} \sum_{x \in \ZZ} \| d_n(x) \|^p_\BB,
\end{align*}
as required.

Next, we consider the case where $k \geq n$. In that case, $\cE_k d_n = 0$, and we need to show 
$\| A_{t_k} d_n \|^p_\lpb \leq c \cdot 2^{n-k} \| d_n \|^p_\lpb$. We will show that for every $x$, we have
\[
 \| (A_{t_k} d_n) (x) \|^p_\BB \leq c \cdot 2^{n - k} \cdot 2^{-k} \cdot \sum_{j < 2^k} \| d_n(x+j) \|^p_\BB.
\]
Summing both sides over $x \in \ZZ$ yields the desired inequality. 

First, we show that for any subset $A$ of $[0, 2^k)$ and any $x$, we have
\begin{equation}
\label{eq:holder}
\Big\| \sum_{j \in A} d_n(x + j) \Big\|^p_\BB \leq 2^{(n+1) \cdot (p - 1)} \sum_{j \in A} \| d_n(x +j) \|^p_\BB.
\end{equation}
Let $D'_h  = [h \cdot 2^n, (h + 1) \cdot 2^n)$. Note that $\sum_{i \in D'_h} d_n(i) = 0$ for every $h$, so we have 
\[
\sum_{j \in A} d_n(x + j) = \sum_{j \in H} d_n(x + j),
\] 
where $H$ is the set of at most $2 \cdot 2^n$ elements at either end of the interval $[x, x + t_k)$ that are not contained in any subinterval $D'_h$ of $[x, x +t_k)$. By H\"older's inequality we have
\begin{align*}
 \Big\| \sum_{j \in A} d_n(x+j) \Big\|_\BB^p & = \Big\| \sum_{j \in A} (d_n(x+j) \cdot 1_H(x+j)) \Big\|_\BB^p \\ 
  & \leq \sum_{j \in A} \| d_n(x+j) \|_\BB^p \cdot | H |^{p-1} \\
  & \leq 2^{(n + 1) \cdot (p-1)} \cdot \sum_{j \in A} \| d_n(x+j) \|_\BB^p.
\end{align*}
So we have
\begin{align*}
 \| (A_{t_k} d_n) (x) \|^p_\BB & = t_k^{-p} \cdot \Big\| \sum_{j < t_k} d_n(x +j) \Big\|^p_\BB \\
  & \leq 2^{-(k-1)p} \cdot 2^{(n + 1) \cdot (p-1)} \cdot \sum_{j < 2^k} \| d_n(x + j) \|_\BB^p \\
  & = 2^{2p-1} \cdot 2^{(n - k) \cdot (p - 1)} \cdot 2^{-k} \cdot \sum_{j < 2^k} \| d_n(x + j) \|_\BB^p \\
  & \leq 2^{2p-1} \cdot 2^{n - k} \cdot 2^{-k} \cdot \sum_{j < 2^k} \| d_n(x + j) \|_\BB^p,
\end{align*}
as required.
\end{proof}

\newtheorem*{lemma:b}{Lemma \ref{lemma:b}}
\begin{lemma:b}
  Let $(t_i)_{i \in \NN}$ be any increasing sequence of natural numbers. Then there is a constant $C$ as in Theorem~\ref{lpb:main:thm} such that
\[
 \sum_k \sum_{i \in S_k} \| A_{t_{i+1}} f - A_{t_i} f \|_\lpb^p \leq C \cdot \| f \|^p_\lpb 
\]
  where $S_k = \{ i \; | \; t_i, t_{i+1} \in [2^{k-1}, 2^k]\}$. 
\end{lemma:b}

\begin{proof}
As in the proof of Lemma~\ref{lemma:a}, it suffices to show
\begin{equation*}
  \sum_{i \in S_k} \| A_{t_{i+1}} d_n - A_{t_i} d_n \|_\lpb^p \leq c \cdot 2^{-|n-k|} \cdot \| d_n \|^p_\lpb
\end{equation*}
because then we have
\begin{align*}
 \sum_k \sum_{i \in S_k} \| A_{t_{i+1}} f - A_{t_i} f \|_\lpb^p & \leq
     \sum_k \sum_{i \in S_k} \big(\sum_{n \in \ZZ} \| A_{t_{i+1}} d_n - A_{t_i} d_n \|_\lpb\big)^p \\
 & = \sum_k \bigg(\Big(\sum_{i \in S_k} \big(\sum_{n \in \ZZ} \| A_{t_{i+1}} d_n - A_{t_i} d_n \|_\lpb\big)^p\Big)^{1/p}\bigg)^p \\
 & \leq \sum_k \bigg(\sum_{n \in \ZZ} \Big(\sum_{i \in S_k} \| A_{t_{i+1}} d_n - A_{t_i} d_n \|_\lpb^p\Big)^{1/p}\bigg)^p \\
 & \leq \sum_k \Big(\sum_{n \in \ZZ} \big(c \cdot 2^{-|n-k|} \cdot \| d_n \|^p_\lpb\big)^{1/p}\Big)^p \\
 & = \sum_k \Big(\sum_{n \in \ZZ} (c \cdot 2^{-|n-k|})^{1/p} \cdot \| d_n \|_\lpb\Big)^p \\
 & \leq \Big(\sum_{n \in \ZZ} (c \cdot 2^{-|n|})^{1/p}\Big)^p \sum_{n \in \ZZ} \| d_n \|^p_\lpb \\
 & \leq c' \cdot \| f \|^p_\lpb,
\end{align*}
where again we use Young's inequality in the second-to-last step.

First, consider the case $n > k$. Set $D_h = [h \cdot 2^{n-1}, (h+1) \cdot 2^{n-1})$ as in the proof of Lemma~\ref{lemma:a}, and set $m_h$ to be the constant value of $\| d_n(x) \|_\BB$ on $D_h$. Then for every $x \in D_h$, we have
\begin{align*}
\sum_{i \in S_k} & \| (A_{t_{i+1}} d_n)(x) - (A_{t_i} d_n)(x) \|_\BB \\
  & \leq \sum_{i \in S_k} \bigg((t^{-1}_i - t^{-1}_{i+1}) \cdot \sum_{j \in [0,t_i)} \| d_n(x +j) \|_\BB \bigg) + \sum_{i \in S_k} \bigg(t^{-1}_{i+1} \cdot \sum_{j \in [t_i, t_{i+1})} \| d_n(x + j) \|_\BB\bigg) \\
& \leq \sum_{i \in S_k} \bigg((t^{-1}_i - t^{-1}_{i+1}) \cdot \sum_{j \in [0, 2^k)} \| d_n(x +j) \|_\BB  \bigg)
  + \sum_{i \in S_k} \bigg( 2^{-(k-1)} \cdot \sum_{j \in [t_i, t_{i+1})} \| d_n(x + j) \|_\BB \bigg) \\
& \leq \bigg( \sum_{i \in S_k} (t^{-1}_i - t^{-1}_{i+1})  \bigg) \cdot \sum_{j \in [0, 2^k)} \| d_n(x +j) \|_\BB  
  + 2^{-(k-1)} \cdot \sum_{j \in [0, 2^k)} \| d_n(x + j) \|_\BB \\
& \leq 2 \cdot 2^{-(k-1)} \cdot 2^k \max(m_h, m_{h+1}) \\
& = 4 \max(m_h, m_{h+1}).
\end{align*}
Moreover, for all but at most $2^k$ values of $x$ such that $[x, x + 2^k)$ extends into $D_{h+1}$, the left-hand side of the preceding inequality is equal to $0$. Hence,
\begin{align*}
\sum_{i \in S_k} \| A_{t_{i+1}} d_n - A_{t_i} d_n \|_\lpb^p & 
  = \sum_{h \in \ZZ} \sum_{x \in D_h} \sum_{i \in S_k} \| (A_{t_{i+1}} d_n)(x) - (A_{t_i} d_n)(x) \|_\BB^p \\
& \leq \sum_{h \in \ZZ} \sum_{x \in D_h} \Big(\sum_{i \in S_k} \| (A_{t_{i+1}} d_n)(x) - (A_{t_i} d_n)(x) \|_\BB \Big)^p \\
& \leq \sum_{h \in \ZZ} 2^k \cdot (4 \max(m_h, m_{h+1}))^p \\
& \leq 2^k \cdot 4^p \cdot \Big(\sum_{h \in \ZZ} m_h^p + \sum_{h \in \ZZ} m_{h+1}^p \Big)\\
& = c \cdot 2^{k-n} \| d_n \|^p_\lpb,
\end{align*}
as required, with $c = 2 \cdot 2 \cdot 4^p$.

Now consider the case $k \geq n$. As in the proof of Lemma~\ref{lemma:a}, it suffices to prove
\[
 \sum_{i \in S_k} \| ((A_{t_{i+1}} - A_{t_i}) d_n)(x) \| ^p \leq c \cdot 2^{n - k} \cdot 2^{-k} \sum_{j < 2^k} \| d_n(x + j) \|^p_\BB, 
\]
because then summing over $x$ yields the desired inequality. For each $x$, we have
\begin{multline*}
 \sum_{i \in S_k} \| ((A_{t_{i+1}} - A_{t_i}) d_n)(x) \| ^p \leq \sum_{i \in S_k} \bigg( (t^{-1}_i - t^{-1}_{i+1}) \cdot \bigg\|  \sum_{j \in [0, t_i)} d_n(x +j) \bigg\|_\BB  \bigg)^p \\
  + \sum_{i \in S_k} \bigg( 2^{-(k-1)} \cdot \bigg\| \sum_{j \in [t_i, t_{i+1})} d_n(x + j) \bigg\|_\BB \bigg)^p.
\end{multline*}
Write the left-hand side as $T_1 +T_2$. Using equation (\ref{eq:holder}), we have
\begin{align*}
 T_1 & \leq \sum_{i \in S_k} \bigg((t^{-1}_i - t^{-1}_{i+1})^p \cdot 2^{(n+1) \cdot (p-1)} \sum_{j \in [0,t_i)} \| d_n(x +j) \|^p \bigg) \\
  & \leq  2^{(n+1) \cdot (p-1)} \cdot \sum_{j < 2^k} \| d_n(x +j) \|^p \cdot \Big(\sum_{i \in S_k} (t^{-1}_i - t^{-1}_{i+1}) \Big)^p\\
  & \leq 2^{(n+1) \cdot (p-1)} \cdot 2^{-(k-1) p} \cdot \sum_{j < 2^k} \| d_n(x +j) \|^p
\end{align*}
Similarly, we have
\begin{align*}
T_2 & \leq 2^{-(k-1)p} \cdot 2^{(n + 1) \cdot (p - 1)} \cdot \sum_{i \in S_k} \sum_{j \in [t_i, t_{i+1})} 
  \| d_n(x +j) \|^p \\
  & \leq 2^{(n+1) \cdot (p-1)} \cdot 2^{-(k-1) p} \cdot \sum_{j < 2^k} \| d_n(x +j) \|^p,
\end{align*}
and by the calculation at the end of the proof of Lemma~\ref{lemma:a}, we are done.
\end{proof}

In the case where $\lpb$ is $\ell^2$, Theorems $A'$ and $B$ in \cite{jones:et:al:03} are more general than our Lemma~\ref{lemma:a} and \ref{lemma:b}, in three senses:
\begin{enumerate}
 \item The sequences $t_k$ can depend on $x$. 
 \item Rather than intervals $[0, t)$ in $\ZZ$, they apply to cubes in $\ZZ^d$, for arbitrary $d$.
 \item The sequence of intervals $([0,t))_{t \in \NN}$ are replaced by any sequence of cubes $(A_t)_{t \in \NN}$ satisfying certain constraints.
\end{enumerate}
Although the definitions and notation become more complex, the proofs in \cite{jones:et:al:03} have essentially the same structure as the ones we have presented here, and can again be adapted to $\lpb$ following the strategies described above. 

Using the first generalization, along with the analogous martingale inequalities, Theorem~\ref{lpb:main:thm} can be strengthened in the following ways:
\begin{theorem}
\label{lpb:stronger:thm}
  Let $p \geq 2$, and let $\BB$ be a Banach space with modulus of uniform convexity $\eta(\varepsilon) = K \varepsilon^p$. Then there is a constant $C$ depending only on $p$ and $K$ such that for any increasing sequence $(t_k)_{k \in \NN}$ of positive natural numbers,  
\[
 \left\| \left(\sum_k \sup_{u, v \in [t_k, t_{k+1}]} \| A_u f - A_v f \|_\BB^p \right)^{1/p} 
 \right\|_{\ell^p} \leq C \cdot \| f \|_\lpb. 
\]
\end{theorem}

\begin{theorem}
\label{lpb:stronger:thm:b}
  Let $q > p \geq 2$, and let $\BB$ be a Banach space with modulus of uniform convexity $\eta(\varepsilon) = K \varepsilon^p$. Then there is a constant $C$ depending only on $p$, $q$, and $K$ such that
\[
 \left\| \left(\sup_{(t_k)} \sum_k \| A_{t_{k+1} f} - A_{t_k} f \|_\BB^q \right)^{1/q} 
 \right\|_{\ell^p} \leq C \cdot \| f \|_\lpb, 
\]
  where the supremum ranges over all increasing sequences $(t_k)$.
\end{theorem}
Note that in each case the expression in parentheses denotes a pointwise supremum, which is to say, a function of $x \in \ZZ$. Theorem~\ref{lpb:stronger:thm} is the $\lpb$ analogue of Theorem 1.8 of \cite{jones:et:al:98} and Theorem $A$ of \cite{jones:et:al:03}, and relies on a Banach-valued version of the martingale inequality given in Theorem 6.1 of \cite{jones:et:al:98}. Theorem~\ref{lpb:stronger:thm:b} is the $\lpb$ analogue of Theorem 1.10 of \cite{jones:et:al:98} and Theorem $B$ of \cite{jones:et:al:03}, and is proved using the Banach-valued version of L\'epingle's martingale inequality given in \cite{pisier:xu:88}. 

The Calder\'on transfer principle \cite{calderon:68} yields the corresponding results for $f \in L^p(X {;\,} \BB)$ and any measure-preserving transformation $T$ of $X$. Each theorem implies that the sequence of ergodic averages $(A_n f)$ converges pointwise~a.e.\ in $L^p(X {;\,} \BB)$, thereby providing strong quantitative versions of the pointwise ergodic theorem.

For the spaces $L^p$, Jones, Kaufman, Rosenblatt, and Wierdl \cite{jones:et:al:98,jones:et:al:03} obtain additional variational inequalities, including weak type (1, 1) inequalities, type $(q,q)$ inequalities for all $1<q<\infty$, and results for $L^\infty$, using variants of the Caleder\'on-Zygmund decomposition, martingale theorems, and a host of other methods. It seems that most of the arguments can be transferred to the setting of $L^p(X {;\,} \BB)$ for arbitrary $p$-uniformly convex Banach spaces $\BB$. (To that end, consider the martingale inequalities for $L^p(X {;\,} \BB)$ in \cite[Chapter 4]{pisier:11}, the analogue of the Calder\'on-Zygmund decomposition \cite[Theorem 8.13]{pisier:11}, and the analogue of the Marcinkiewicz interpolation theorem \cite[Theorem 8.51]{pisier:11}.) Pursuing this here, however, would take us too far afield.

\section{A weaker result for nonexpansive operators}
\label{old:main:section}

Fix $p \geq 2$ and a $p$-uniformly convex Banach space $\BB$ with modulus of uniform convexity $\eta(\varepsilon) = K\varepsilon^p$. In this section, we obtain a weaker result, in the case where $T$ is a nonexpansive operator that is not necessarily power bounded from below.

\newtheorem*{old:main:theorem}{Theorem~\ref{old:main:thm}}
\begin{old:main:theorem}
Suppose $T$ is a nonexpansive linear operator on $\BB$. Then for any $x$ in $\BB$ and any $\varepsilon > 0$, there are at most $C \rho^{p + 1} \log \rho$-many $\varepsilon$ fluctuations in $(A_n x)$, for a constant $C$ that depends only on $p$ and $K$.
\end{old:main:theorem}
\noindent This should be compared to the bound of $C \rho^p$ obtained from Theorem~\ref{new:main:thm}.

Let $\tilde \eta(\varepsilon) = K \varepsilon^{p-1}$. For any element $x$ of $\BB$, we follow the notation of Kohlenbach and Leu\c{s}tean \cite{kohlenbach:leustean:09} closely by writing $x_n$ for the average $A_n x$. The following lemma is implicit in that paper:

\begin{lemma}
\label{lemma:one}
For any $x \in \BB$, let
\[
M = \left\lceil \frac{16\| x \|}{\varepsilon} \right\rceil, \quad \gamma = \frac{\varepsilon}{8} \tilde \eta \left( \frac{\varepsilon}{8 \| x\| } \right).
\]
Suppose that $N$ and $u$ are such that for every $m \leq u$, $\| x_m \| \geq \| x_N \| - \gamma$. Then for every pair $i, j$ in the interval $[MN, \lfloor u / 2 \rfloor]$, we have $\|x_i - x_j\| < \varepsilon$.
\end{lemma}

\begin{proof}
 This is exactly the calculation in Section 4 of Kohlenbach and Leu\c{s}tean \cite{kohlenbach:leustean:09}, with $h(N)$ replaced by $u$, $g(MN)$ replaced by $u/2 - MN$, and $b$ replaced by $\|x\|$. 
\end{proof}

Notice in particular that if $\| x_N \|$ is a ``global $\gamma$-minimum,'' which is to say, $\| x_m \| \geq \| x_N \| - \gamma$ for every $m$, then Lemma~\ref{lemma:one} implies $\|x_i - x_j\| < \varepsilon$ for all $i, j > MN$.

The next lemma shows that, in a certain sense, a small interval cannot contain many $\varepsilon$-fluctuations.

\begin{lemma}
\label{lemma:two}
 Fix $N \geq 1$ and real numbers $\alpha \geq 1$ and $\varepsilon > 0$. Let $x$ be any element of $\BB$ such that $\varepsilon < 2 \| x \|$. Then the number of $\varepsilon$-fluctuations between $N$ and $\alpha N$ is at most $\lfloor 4 \log \alpha \cdot \| x \| / \varepsilon \rfloor$. 
\end{lemma}

\begin{proof}
Suppose
\[
 N \leq i_1 \leq j_1 \leq i_2 \leq j_2 \leq \ldots \leq i_s \leq j_s \leq \alpha N
\]
satisfy $\|x_{j_u} - x_{i_u}\| \geq \varepsilon$ for every $u$. We need to show $s \leq 4 \log \alpha \cdot \| x \| / \varepsilon$.

A straightforward calculation (see equation (6) of \cite{birkhoff:39} or equation (11) of \cite{kohlenbach:leustean:09}) shows that for every $n, k \geq 1$, $\| x_{n+k} - x_n \| \leq 2 k \| x \| / (n + k)$. In particular, for any $j = j_u$ and $i = i_u$, we have $2(j-i)\|x\| / j \geq \varepsilon$, and so
\[
 j \geq \left(\frac{2\|x\|}{2\|x\| - \varepsilon}\right) \cdot i = \left(1 + \frac{\varepsilon}{2\|x\| - \varepsilon}\right) \cdot i > \left(1 + \frac{\varepsilon}{2\|x\|} \right) \cdot i.
\]
Since $i_1 \geq N$, we have $j_1 \geq (1 +  \varepsilon / (2\|x\|)) \cdot N$; since $i_2 \geq j_1$, we have $j_2 \geq (1 + \varepsilon / (2\|x\|))^2 \cdot N$, and so on. Thus $j_s \geq (1 +  \varepsilon / (2\|x\|))^s \cdot N$. Since $j_s \leq \alpha N$, we have
\[
 (1 +  (\varepsilon / (2\|x\|))^s \leq \alpha
\]
and hence 
\[
 s \leq \frac{\log \alpha}{\log (1 + \varepsilon / (2\|x\|))} < 4 \log \alpha \cdot \| x \| / \varepsilon,
\]
since $\log (1 + \varepsilon / (2\|x\|)) > \varepsilon / (4\|x\|)$ when $\varepsilon <2 \| x \|$.
\end{proof}

Suppose we are given $x$ in $\BB$ such that $\varepsilon < 2 \| x \|$. Consider the sequence $x_1, x_2, x_3, \ldots$ of averages. Let $M$ and $\gamma$ be as in the statement of Lemma~\ref{lemma:one}. Define a finite sequence $N_0, N_1, \ldots, N_s$ of integers by setting $N_0 = 1$ and setting $N_{i+1}$ equal to the least $m$ such that $\| x_m \| < \| x_{N_i} \| - \gamma$, if such an $m$ exists. Notice that $N_{i+1} \geq N_i$ for every $i < s$. Notice also that $s \leq \lfloor \| x \| / \gamma \rfloor$, as the norm of the averages cannot drop by $\gamma$ more than $\lfloor \| x \| / \gamma \rfloor$-many times.

The idea is this: Lemma~\ref{lemma:one} tells us that there are no $\varepsilon$-fluctuations in the intervals $[MN_0, N_1/ 2)$, $[MN_1, N_2/2)$, \ldots, $[MN_{s-1}, N_s/2)$, or beyond $MN_s$. This leaves the $\varepsilon$-fluctuations in the intervals $[1,MN_0)$ and $[N_u/2,MN_u)$ for 
$u=1,\ldots,s$, whose number we can bound using Lemma~\ref{lemma:two}; as well as at most $s$-many $\varepsilon$-fluctuations that span more than one interval. (The fact that the intervals in the first sentence may overlap or that some of the $N_u$'s may be odd does not hurt the argument below.)

More precisely, suppose $i_1 \leq j_1 \leq \ldots \leq i_k \leq j_k$ are such that for each $u=1,\ldots,k$, $\| a_{j_u} - a_{i_u}\| \geq \varepsilon$. Lemma~\ref{lemma:two} tells us that at most $\lfloor 4 \log M \cdot \| x \| / \varepsilon \rfloor$ of the pairs $(i_u,j_u)$ lie in the first interval, $[1,MN_0) = [1,M)$, and at most $\lfloor 4 \log (2 M) \cdot \| x \| / \varepsilon \rfloor$ of them lie in the remaining ones. Each of the remaining pairs has to straddle at least one of the $N_u$'s for $u=1,\ldots,s$. Thus $k$ is at most
\[
\left\lfloor 4 \log M \cdot \frac{\|x\|}{\varepsilon} \right\rfloor + \left\lfloor \frac{\| x \|}{\gamma} \right\rfloor \cdot \left\lfloor 4 \log (2M) \cdot \frac{\|x\|}{\varepsilon}\right\rfloor + \left\lfloor \frac{\| x \|}{\gamma} \right\rfloor.
\]
This provides a precise bound on the number of fluctuations, but some simplification will improve readability.

If $\| x \| / \varepsilon$ is sufficiently large, we can expand the definitions of $M$ and $\gamma$ and absorb the first and third terms and various constants into a constant multiple of the second term. More precisely, setting $\rho = \| x \| / \varepsilon$, the sequence of ergodic averages admits at most $O(\rho^2 \log \rho \cdot \tilde \eta(1/(8\rho))^{-1})$-many $\varepsilon$-fluctuations. Expanding the definition of $\tilde \eta$, we obtain a bound of $O(\rho^{p +1} \log \rho)$, completing the proof of Theorem~\ref{old:main:thm}.

\section{Quantitative convergence theorems}
\label{quantitative:section}

Let $(a_n)$ be a sequence of elements of a complete metric space. The next three statements all express the fact that $(a_n)$ is convergent:
\begin{enumerate}
 \item For every $\varepsilon > 0$, there is an $n$ such that for every $i, j \geq n$, $d(a_i,a_j) < \varepsilon$.
 \item For every $\varepsilon > 0$, there is a $k$ such that $(a_n)$ admits at most $k$ $\varepsilon$-fluctuations.
 \item For every $\varepsilon > 0$ and function $g(n)$, there is an $n$ such that for every $i, j \in [n, g(n)]$, $d(a_i,a_j) < \varepsilon$.
\end{enumerate}
Even though the statements are equivalent, the existence assertions are quite different. A bound $r(\varepsilon)$ on the value of $n$ as in (1) is called a \emph{bound on the rate of convergence} of $(a_n)$. We will call a bound $s(\varepsilon)$ on $k$ as in (2) a \emph{bound on the number of $\varepsilon$-fluctuations}, and a bound $t(\varepsilon,g)$ on $n$ as in (3) a \emph{bound on the rate of metastability}.

Notice that any bound on the rate of convergence of $(a_n)$ provides, \emph{a fortiori}, a bound on the number of $\varepsilon$-fluctuations. Moreover, if, for every $\varepsilon > 0$, $s(\varepsilon)$ is a bound on the number of $\varepsilon$-fluctuations, then for any monotone function $g$, one of the intervals
\[
[1, g(1)], [g(1), g^2(1)], \ldots, [g^{s(\varepsilon)}(1),g^{s(\varepsilon)+1}(1)]
\]
must fail to include a pair $i, j$ with $d(a_i,a_j) > \varepsilon$. Hence $t(\varepsilon,g) = g^{s(\varepsilon)}(1)$ is a bound on the rate of metastability.

On the other hand, it is well known in the general study of computability that there are computable, bounded, increasing sequences of rationals $(a_n)$ that fail to have a computable rate of convergence. (Such a sequence is called a Specker sequence; see, for example, \cite{pourel:richards:89} or the discussion in \cite[Section 5]{avigad:et:al:10}.) Clearly for such a sequence there is a computable bound on the number of fluctuations. Similarly, it is not hard to construct a computable, bounded sequence of rationals for which there is no computable bound on the number of fluctuations. (Roughly speaking, have the sequence oscillate $n$ times by some $\varepsilon_n$ whenever the $n$th Turing machine is seen to halt on empty input.) But as long as $g$ is computable, one can always compute a bound on $t(\varepsilon, g)$ by searching for a suitable interval. In a similar way, it is not hard to construct classes of sequences with a uniform bound on the number of fluctuations, but no uniform bound on the rate of convergence; and classes of sequences with a uniform bound on the rate of metastability, but no uniform bound on the number of fluctuations.\footnote{As an example of the latter, consider the countable collection of sequences where the $j$th sequence starts out with $0$'s and then oscillates $j$-times from $0$ to $1$ or back at the $j$th element. For any function $g$, if $g(1) < j$, then $[1,g(1)]$ has no oscillations; otherwise, $g(1) \geq j$ and one of the intervals $[g(1), g^2(1)], \ldots, [g^{g(1)+1}(1),g^{g(1)+2}(1)]$ has no oscillations. Thus $t(g,\varepsilon) = g^{g(1)+1}(1)$ is a uniform bound on the rate of metastability.}

Now consider the mean ergodic theorem, say, for a nonexpansive linear operator on a Hilbert space. It has long been known \cite{krengel:78} that there is no uniform bound on the rate of convergence, and it is not hard to show \cite{avigad:simic:06,avigad:et:al:10,avigad:12,vyugin:97,vyugin:98} that one cannot generally compute a bound on the rate of convergence from the given data. (Avigad, Gerhardy, and Towsner \cite{avigad:et:al:10} show, however, that in the case of a Hilbert space, one \emph{can} compute a bound on the rate of convergence of the ergodic averages, \emph{given} the norm of the limit. The considerations here show that this result extends to uniformly convex Banach spaces more generally. Specifically, for every $n, k \geq 1$ we have $\| A_{kn} f \| = \| \frac{1}{k} \sum_{i<k} T^{in} A_n \| \leq \| A_n f \|$, and hence the norm of the ergodic limit is the infimum of the norms $\| A_n f \|$. Lemma~\ref{lemma:one} above and the comment after the proof then shows that one can compute a rate of convergence by waiting until the norm of one of the averages is sufficiently close to this infimum.)

The variational inequalities in Theorems~\ref{jro:thm}, \ref{thm:jro:mean}, \ref{new:main:thm}, and \ref{old:main:thm} all yield uniform and explicit bounds on the number of $\varepsilon$-fluctuations in a sequence of ergodic averages, and hence on the rate of metastability. As such, these strengthen the results of Avigad, Gerhardy, and Towsner \cite{avigad:et:al:10} and Kohlenbach and Leu\c{s}tean \cite{kohlenbach:leustean:09}. In the case of a nonexpansive map on a uniformly convex Banach space, however, the more direct argument by Kohlenbach and Leu\c{s}tean \cite{kohlenbach:leustean:09} yields a bound on the rate of metastability that is quantitatively better, requiring only $O(\rho \log \rho \cdot \eta(1/(8\rho))^{-1})$-many iterations of a function that grows slightly faster than the argument $g$ described above.\footnote{Note that due to an error in typesetting there is an extra ``$h$'' in the statement of the main theorem in \cite{kohlenbach:leustean:09}.}

Tao \cite{tao:08} and Walsh \cite{walsh:12} use metastability to establish the norm convergence of more complex forms of ergodic averages. Kohlenbach, Leu\c{s}tean and Schade have since obtained uniform bounds on the rate of metastability in much more general settings \cite{kohlenbach:leustean:09,kohlenbach:unp:e,kohlenbach:leustean:unp,kohlenbach:schade:12}. Gerhardy and Kohlenbach \cite{gerhardy:kohlenbach:08} show that under very general conditions, having to do with derivability in a certain (strong) axiomatic theory, there are uniform and computable bounds on rates of metastability. Avigad and Iovino \cite{avigad:iovino:unp} show, again in a very general setting, that the closure of the class of structures in question under the formation of ultraproducts is enough to guarantee uniformity.

\section{Lower bounds}
\label{lower:bounds:section}

In this section, we show that the upper bound given by Theorem~\ref{new:main:thm} is sharp. In fact, we prove something stronger, namely, that the resulting bounds on the rate of metastability, and hence the bounds on the number of $\varepsilon$-fluctuations, are sharp as well.

Consider the complex numbers $\mathbb{C}$ as a Banach space over the reals, with isometry $T_\theta(z) = e^{i \theta} \cdot z$. For every $n$, $A_n 1 = (e^{i n \theta} -1) / (n (e^{i \theta} - 1))$. So, in particular, for $\theta = \pi / k$, we have
\[
 | A_k 1 | =\left| \frac{2}{k (e^{\pi i / k} - 1)} \right| \geq 2 / \pi \geq 1 / 2,
\]
since $| e^{i \theta} - 1 | \leq \theta$ for every $\theta$. Moreover, $A_{2k} 1 = 0$, which is to say, there is a $1/2$-fluctuation in the sequence $(A_n 1)$ between $k$ and $2k$.

For each $u$ and $p \geq 2$ we consider the space $\ell^p_u(\mathbb{C})$ of $u$-tuples of complex numbers with norm
\[
 \| (z_1, \ldots, z_u) \|_{p,u} = (| z_1 |^p + \ldots + | z_u |^p)^{1 / p}.
\]
Once again, we view this as a Banach space over the reals. Note that for each $i$ we have $\| (z_1, \ldots, z_u) \|_{p,u} \geq | z_i |$. Clarkson's inequalities apply equally well to the complex-valued function spaces (see e.g.~\cite{adams:fournier:03}), so for each $p \geq 2$, $\ell^p_u(\mathbb{C})$ is uniformly convex with greatest modulus of uniform convexity 
\[
 \eta(\varepsilon) = 1 - (1 - (\varepsilon / 2)^p)^{1/p} \geq \frac{1}{p} (\varepsilon / 2)^p.
\]
Let ${\bf 1} = u^{-1/p} \cdot (1, 1, \ldots, 1)$, so $\| {\bf 1} \|_{p, u} = 1$. Let ${\bf T} = (T_{\pi}, T_{\pi/2}, T_{\pi/4}, \ldots, T_{\pi/2^{u-1}})$. By the analysis above, the sequence $(A_n {\bf 1})$ of ergodic averages corresponding to ${\bf T}$ will have a $1 / (2 u^{1/p})$-fluctuation in each interval $[1,2], [2,4], \ldots, [2^{u-1},2^u]$, that is, $u$-many fluctuations in all. Setting $\varepsilon = 1 / (2 u^{1/p})$ and thinking of $u$ as a function of $\varepsilon$, we have obtained the following:

\begin{theorem}
Let $p \geq 2$, $n_p(\varepsilon) = \frac{1}{p}(\frac{\varepsilon}{2})^p$. Then for every $\varepsilon > 0$ there is a Banach space $\BB$ with modulus of uniform convexity $\eta_p$, an isometry $T$, and an $x$ in $\BB$ with $\| x \| = 1$, such that for $u = \lfloor (2 \varepsilon)^{-p} \rfloor$, the sequence $(A_n x)$ has an $\varepsilon$-fluctuation in each interval $[1,2], [2, 4], \ldots, [2^{u-1},2^u]$.
\end{theorem}

This provides a lower bound of $\Omega( \rho^p)$ $\varepsilon$-fluctuations for sequences $(A_n x)$ in uniformly convex Banach spaces with modulus of uniform convexity $\eta_p$, where $\rho = \| x \| / \varepsilon$. This complements the upper bound of $O(\rho^p)$ provided by Theorem~\ref{new:main:thm}.

In particular, setting $\varepsilon = 1 / 4$, we see that there are $2^p$-many $(1/4)$-fluctuations in the sequence $(A_n {\bf 1})$ in $\ell^p_{2^p}(\mathbb{C})$. Thus we have:
\begin{corollary}
\label{lb:one}
Let $\eta_p(\varepsilon) = \frac{1}{p} (\varepsilon / 2)^p$ and suppose $t(\varepsilon, g)$ is any bound on the rate of metastability for sequences of ergodic averages $(A_n x)$ with $\| x \| \leq 1$, for Banach spaces with modulus of uniform convexity $\eta_p$. Then for $g(n) = 2^{\lceil \log_2 n \rceil + 1}$, $t(1/4, g) \geq 2^p$. Similarly, if $s(\varepsilon)$ is any bound on the number of $\varepsilon$-fluctuations in such sequences, $s(1/4) \geq 2^p$ as well.
\end{corollary}
This shows that any bound on the rate of metastability for uniformly convex Banach spaces has to depend on the modulus of uniform convexity, and similarly for the number of $\varepsilon$-fluctuations. The next counterexample shows that, without the hypothesis of uniform convexity, one can find counterexamples to the uniformity in a single space.

\begin{corollary}
There is a separable, reflexive, and strictly convex Banach space $\BB$, such that for every $u$, there is an $x \in \BB$ with $\| x \| \leq 1$ such that the sequence $(A_n x)$ has $(1/4)$-fluctuations in each of the intervals $[1,2], [2,4], \ldots, [2^{u-1}, 2^u]$.
\end{corollary}

\begin{proof}
For each $u$, let $\BB_u$ be any separable, uniformly convex Banach space containing an $x$ with the requisite properties. As in Day \cite{day:41}, define $\BB$ to be the space of sequences $b = (b_u)$ such that each $b_u$ is in $\BB_u$ and $\| b \| = (\sum_u \| b_u \|^2_u)^{1/2} < \infty$. Then $\BB$ embeds each $\BB_u$, and Day \cite{day:41} shows that $\BB$ is reflexive and strictly convex.
\end{proof}

We also provide a partial converse to Theorem~\ref{new:main:thm}.  A similar analysis can be carried out for $\varepsilon$-fluctuations and metastability.

\begin{theorem}
\label{partial:converse:thm}
Consider the following conditions on a Banach space $\BB$ for some $p \geq 2$.
\begin{enumerate}  
\item $\BB$ is isomorphic to a $p$-uniformly convex space.
\item For every $\lambda \geq 1$ there is a constant $C_\lambda$, such that the following holds.  Assume $\mathbb{X}$ is a closed subspace of $\BB$ and that $T$ is a linear operator on $\mathbb{X}$ satisfying $\lambda^{-1} \|y\|_\BB \leq \|T^n y\|_\BB \leq \lambda \|y\|_\BB$ for every $y$ in $\mathbb{X}$ and every $n$.  Then for all $x$ in $\mathbb{X}$, 
\[ \sum_k \|A_{2^{k+1}} x - A_{2^k} x\|_\BB^p \leq C_\lambda \| x \|_\BB^p.\]
\item $\BB$ has type $r > 1$ and cotype $p$  (see \cite{maurey:03, pisier:86} for definitions).
\end{enumerate}
In general, (1) implies (2) implies (3).  Further, if $\BB$ is a UMD space or a Banach lattice, then the three conditions are equivalent.
\end{theorem}

\begin{proof} The first implication follows from Theorem~\ref{new:main:thm}, since a modulus of uniform convexity remains valid on any subspace.  For the second implication, first assume that $\BB$ is not of cotype $p$ for some $p\geq2$ (and therefore infinite dimensional).  Then $\BB$ is also not of cotype $q$ for some $q>p$ since each space has a minimal cotype \cite{maurey:03}.  

Consider the isometry $\mathbf{T}$ and point $x=\mathbf{1}$ on $\ell^q_u(\CC)$ given above with $\varepsilon = 1 / (2 u^{1/q})$. By the same analysis, for each $k < u$, $\|A_{2^{k+1}} x - A_{2^k} x\|_{q,u} \geq \varepsilon$.  It follows that
\[ \sum_k \|A_{2^{k+1}} x - A_{2^k} x\|_{q,u}^p \geq u \cdot \varepsilon^p = \frac{u^{1-p/q}}{2^p}\] 
which is unbounded in $u$ since $p/q < 1$.  Therefore, it is enough to find $\lambda$-isomorphic copies of $\ell^q_u(\CC)$ in $\BB$ for all $u$.

To do this, first note that $\ell^q_u(\CC)$ is isomorphic to $\ell^q_{2u}$.  Indeed, for $a,b$ in $\RR$, 
\[(a^q + b^q)^{1/q} \leq (a^2 + b^2)^{1/2}  \leq \sqrt{2} \max(|a|,|b|) \leq \sqrt{2} (a^q + b^q)^{1/q},\]
and in the same way, for $z_j = a_j +i \cdot b_j$,
\[ \|(z_1, \ldots, z_u)\|_{\ell^q_u(\CC)} \leq \sqrt{2} \|(a_1, b_1, \ldots, a_u, b_u)\|_{\ell^q_{2u}} \leq \sqrt{2} \|(z_1, \ldots, z_u)\|_{\ell^q_u(\CC)}.\]

Second, since $\BB$ is not of type $q$, it follows from a result of Maurey, Pisier and Krivine \cite{maurey:03} that for any $\varepsilon >0$, $\BB$ contains $(1+\varepsilon)$-isomorphic copies of $\ell^q_u$. This means that there are $x_1, \ldots, x_n$ in $\BB$ such that for any $a_1, \ldots, a_u$ in $\RR$,
\[ \|(a_1, \ldots, a_u)\|_{\ell^q_u} \leq \|a_1 x_1 + \cdots + a_u x_u\|_\BB \leq (1+\varepsilon) \|(a_1, \ldots, a_u)\|_{\ell^q_u}.\]

Fix $\lambda > \sqrt{2}$. Then for any $u$, we can find a subspace $\mathbb{X}$ of $\BB$ $\lambda$-isomorphic to $\ell^q_u(\CC)$.  As above, there are an $x$ in $\mathbb{X}$ with $\|x\| \leq \lambda$ and a linear operator $T$ on $\mathbb{X}$ satisfying $\lambda^{-1} \|y\|_\BB \leq \|T^n y\|_\BB \leq \lambda \|y\|_\BB$ for all $y$ in $\mathbb{X}$ and all $n$, such that  
\[ \sum_k \|A_{2^{k+1}} x - A_{2^k} x\|_\BB^p \geq \frac{u^{1-p/q}}{\lambda \cdot 2^p}.\]
Therefore there can be no such $C_\lambda$.

Now assume $\BB$ has type $r = 1$.  By the same result of Maurey, Pisier and Krivine \cite{maurey:03}, for any $\varepsilon >0$, $\BB$ contains $(1+\varepsilon)$-isomorphic copies of $\ell^1_u$.  As before it is enough to find some $x$ of unit norm and isometry $T$ on $\ell^1_u$ where the variation sum with exponent $p$ is unbounded in $u$.

Consider the right shift isometry $T$ (with wrapping) on $\ell^1_u$ and the point $x = (1, 0,\ldots, 0)$.  Then $A_2 x = (\frac{1}{2}, \frac{1}{2}, 0, \ldots, 0)$, $A_3 x = (\frac{1}{3}, \frac{1}{3}, \frac{1}{3}, 0 \ldots, 0)$, and so on.  Notice that, $A_4 - A_2 = (-\frac{1}{4}, -\frac{1}{4}, \frac{1}{4}, \frac{1}{4}, 0, \ldots, 0)$.  In the same way, $\|A_{2^{k+1}} x - A_{2^k} x\|_{1,u} = 1$ for all $k < \lfloor \log_2 u \rfloor$.  Hence
\[ \sum_k \|A_{2^{k+1}} x - A_{2^k} x\|_{1,u}^p \geq \lfloor \log_2 u \rfloor\]
which is clearly unbounded in $u$. This completes the the second implication.

Finally, if $\BB$ is a UMD space, then $\BB$ is of cotype $p$ if and only if $\BB$ is isomorphic to a $p$-uniformly convex space \cite{pisier:86}.   The same holds of Banach lattices of type $r > 1$  \cite[ch.~1-f]{lindenstrauss:79}.  This proves the equivalence.  
\end{proof}

\section{Comments and questions}
\label{upcrossing:section}

We have observed that saying that a sequence of elements of a complete metric space converges is equivalent to saying that, for every $\varepsilon > 0$, the sequence admits only finitely many $\varepsilon$-fluctuations. Similarly, if $(f_n)$ is a sequence of measurable functions from a measure space $\mdl X = (X, \BB, \mu)$ to some metric space, saying that $(f_n)$ converges pointwise a.e.~is equivalent to saying that for every $\varepsilon > 0$, the measure of the set
\[
 \{ x \in X \st \mbox{$(f_n(x))_{n \in \NN}$ admits $k$ $\varepsilon$-fluctuations} \}
\]
approaches $0$ as $k$ approaches infinity. Such results can often be obtained from upcrossing inequalities in the style of Doob's upcrossing inequality for the martingale convergence theorem \cite{doob:53} and Bishop's upcrossing inequalities for the pointwise ergodic theorem and Lebesgue's theorem \cite{bishop:66,bishop:67,bishop:67b}. Other upcrossing inequalities and oscillation inequalities have been obtained in the measure-theoretic setting \cite{ivanov:96,kachurovskii:96,jones:et:al:98, kalikow:weiss:99,jones:et:al:03, hochman:09, oberlin:et:al:12}. 

Theorems~\ref{new:main:thm} and \ref{old:main:thm} show that the uniform bound on the oscillations of a sequence of ergodic averages is a geometric rather than a metric phenomenon. To summarize the state of affairs:
\begin{itemize}
 \item For nonexpansive operators on a Hilbert space, and for linear operators on a uniformly convex Banach space that are power bounded from above and below, variational inequalities yield uniform and explicit bounds on the number of $\varepsilon$-fluctuations in a sequence of ergodic averages (Theorems~\ref{jro:thm} and \ref{new:main:thm}).
 \item For nonexapansive maps on uniformly convex Banach spaces, there are again uniform and explicit bounds on the number of $\varepsilon$-fluctuations (Theorem~\ref{old:main:thm}).
\end{itemize}
The following questions remain:
\begin{itemize}
 \item Can one extend our main result, Theorem~\ref{new:main:thm}, to arbitrary power bounded operators, or even nonexpansive operators?
 \item Can one improve the bound in Theorem~\ref{old:main:thm}?
 \item Can one prove variational inequalities, or obtain uniform bounds on the number of $\varepsilon$-fluctuations, for the sequences of multiple ergodic averages in Tao's and Walsh's theorems \cite{tao:08, walsh:12}? (See \cite{demeter:07,do:et:al:unp} for variational inequalities involving certain kinds of bilinear ergodic averages.)
 \item Does Theorem~\ref{new:main:thm} characterize spaces isomorphic to p-uniformly convex spaces, as does the martingale property in Theorem~\ref{pisier:thm}?
\end{itemize}
Finally, it is worth noting that Kohlenbach's ``proof mining'' program \cite{kohlenbach:08,gerhardy:kohlenbach:08} provides general logical methods for extracting bounds on the rate of metastability, and both those methods and the ultraproduct methods of \cite{avigad:iovino:unp} provide general conditions that guarantee that such bounds are uniform. The methods of \cite{kohlenbach:08,gerhardy:kohlenbach:08} moreover guarantee that the bounds are computable from the relevant data. Along these lines, it would be nice to have a better general understanding as to when (and how) uniform and computable bounds on the number of fluctuations can be obtained from a nonconstructive convergence theorem. (Safarik and Kohlenbach \cite{safarik:kohlenbach:unp} provide some initial results in that direction.)


\begin{thebibliography}{10}

\bibitem{adams:fournier:03}
Robert A.~Adams and John J.~F.~Fournier.
\newblock {\em Sobolev spaces.}
\newblock Elsevier/Academic Press, Amsterdam, second edition, 2003.

\bibitem{avigad:12}
Jeremy Avigad.
\newblock Uncomputably noisy ergodic limits.
\newblock {Notre Dame J. Formal Logic}, 53:347--350, 2012.

\bibitem{avigad:iovino:unp}
Jeremy Avigad and Jos\'e Iovino.
\newblock Ultraproducts and metastability.
\newblock arXiv:1301.3063.

\bibitem{avigad:et:al:10}
Jeremy Avigad, Philipp Gerhardy, and Henry Towsner.
\newblock Local stability of ergodic averages.
\newblock {\em Trans. Amer. Math. Soc.}, 362(1):261--288, 2010.

\bibitem{avigad:simic:06}
Jeremy Avigad and Ksenija Simic.
\newblock {Fundamental notions of analysis in subsystems of second-order
  arithmetic}.
\newblock {\em Ann. Pure Appl. Logic}, 139(1-3):138--184, 2006.

\bibitem{birkhoff:39}
Garrett Birkhoff.
\newblock The mean ergodic theorem.
\newblock {\em Duke Math. J.}, 5(1):19--20, 1939.

\bibitem{bishop:66}
Errett Bishop.
\newblock {An upcrossing inequality with applications}.
\newblock {\em Michigan Math. J.}, 13:1--13, 1966.

\bibitem{bishop:67}
Errett Bishop.
\newblock {\em {Foundations of Constructive Analysis}}.
\newblock McGraw-Hill, New York, 1967.

\bibitem{bishop:67b}
Errett Bishop.
\newblock {A constructive ergodic theorem}.
\newblock {\em J. Math. Mech.}, 17:631--639, 1967/1968.

\bibitem{calderon:68}
A.~P.~Calder\'on.
\newblock Ergodic theory and translation-invariant operators.
\newblock {\em Proc.\ Nat.\ Acad.\ Sci.\ U.S.A.}, 59:349--353, 1968.

\bibitem{carothers:05}
N.~L.~Carothers.
\newblock {A short introduction to Banach space theory}.
\newblock Cambridge University Press, Cambridge, 2005.

\bibitem{day:41}
Mahlon M. Day.
\newblock {Reflexive Banach spaces not isomorphic to uniformly convex spaces}.
\newblock {\em Bull. Amer. Math. Soc.}, 47(4):313--317, 1941.

\bibitem{demeter:07}
Ciprian Demeter.
\newblock Pointwise convergence of the ergodic bilinear Hilbert transform.
\newblock {\em Illinois J. Math.}, 51(4):1123--1158, 2007.
 
\bibitem{doob:53}
J.~L. Doob.
\newblock {\em Stochastic processes}.
\newblock John Wiley \& Sons Inc., New York, 1953.

\bibitem{do:et:al:unp}
Y.~Do, R.~Oberlin, and E.~A. Palsson.
\newblock Variational bounds for a dyadic model of the bilinear Hilbert transform.
\newblock To appear in the \emph{\em Illinois J. Math.}, arXiv:1203.5135.

\bibitem{gerhardy:kohlenbach:08}
Philipp Gerhardy and Ulrich Kohlenbach.
\newblock General logical metatheorems for functional analysis.
\newblock {\em Trans. Amer. Math. Soc.},
  360(5):2615--2660, 2008.

\bibitem{hochman:09}
Michael Hochman.
\newblock Upcrossing inequalities for stationary sequences and applications.
\newblock {\em Ann. Probability}, 37(6):2135--2149, 2009.

\bibitem{ivanov:96}
V.~V. Ivanov.
\newblock {Oscillations of averages in the ergodic theorem}.
\newblock {\em Dokl. Akad. Nauk}, 347(6):736--738, 1996.

\bibitem{jones:et:al:96}
Roger~L. Jones, Iosif V.~Ostrovskii, and Joseph~M. Rosenblatt.
\newblock Square functions in ergodic theory.
\newblock {\em Ergodic Theory Dynam. Systems}, 16(2):267--305.

\bibitem{jones:et:al:98}
Roger~L. Jones, Robert Kaufman, Joseph~M. Rosenblatt, and M\'{a}t\'{e} Wierdl.
\newblock {Oscillation in ergodic theory}.
\newblock {\em Ergodic Theory Dynam. Systems}, 18(4):889--935, 1998.

\bibitem{jones:et:al:03}
Roger~L. Jones, Joseph~M. Rosenblatt, and M\'{a}t\'{e} Wierdl.
\newblock {Oscillation in ergodic theory: higher dimensional results}.
\newblock {\em Israel. J. Math.}, 135:1--27, 2003.

\bibitem{kachurovskii:96}
A.~G. Kachurovski{\u\i}.
\newblock {Rates of convergence in ergodic theorems}.
\newblock {\em Uspekhi Mat. Nauk}, 51(4(310)):73--124, 1996.
\newblock Translation in \emph{Russian Math. Surveys} 51 (1996), no. 4, 653--703.

\bibitem{kalikow:weiss:99}
Steven Kalikow and Benjamin Weiss.
\newblock {Fluctuations of ergodic averages}.
\newblock {\em Illinois J. Math.}, 43(3):480--488, 1999.

\bibitem{kohlenbach:08}
Ulrich~Kohlenbach.
\newblock {\em {Applied proof theory: proof interpretations and their use in
  mathematics}}.
\newblock Springer, Berlin, 2008.

\bibitem{kohlenbach:unp:e}
Ulrich Kohlenbach.
\newblock {A uniform quantitative form of sequential weak compactness and
  Baillon's nonlinear ergodic theorem}.
\newblock {\em Communications in Contemporary Mathematics}, 14, 20 pages, 2012.

\bibitem{kohlenbach:leustean:unp}
Ulrich Kohlenbach and Laurentiu Leu\c{s}tean.
\newblock Effective metastability of {H}alpern iterates in {CAT(0)} spaces. 
\newblock {\em Advances in Math.}, 231:2526--2556, 2012.

\bibitem{kohlenbach:leustean:09}
Ulrich Kohlenbach and Laurentiu Leu\c{s}tean.
\newblock A quantitative mean ergodic theorem for uniformly convex {B}anach
  spaces.
\newblock {\em Ergodic Theory Dynam. Systems}, 29(6):1907--1915, 2009.
\newblock Erratum: \emph{Ergodic Theory Dynam. Systems} 29:1995, 2009.

\bibitem{safarik:kohlenbach:unp}
Ulrich Kohlenbach and Pavol Safarik.
\newblock Fluctuations, effective learnability and metastability in analysis.
\newblock {\em Ann. Pure Appl. Logic}, to appear.

\bibitem{krengel:78}
Ulrich Krengel.
\newblock {On the speed of convergence in the ergodic theorem}.
\newblock {\em Monatsh. Math.}, 86(1):3--6, 1978/79.

\bibitem{krengel:85}
Ulrich Krengel.
\newblock {\em {Ergodic theorems}}.
\newblock Walter de Gruyter \& Co., Berlin, 1985.

\bibitem{lindenstrauss:79}
Joram Lindenstrauss and Lior Tzafriri.
\newblock {\em Classical Banach spaces II: function spaces}.
\newblock Springer, Berlin, 1979.

\bibitem{maurey:03}
Bernard Maurey.
\newblock Type, cotype and $K$-convexity.
\newblock In {\em Handbook of the geometry of Banach spaces}, Vol. 2, North-Holland, Amsterdam, 2003,
1299--1332.

\bibitem{oberlin:et:al:12}
Richard Oberlin, Andreas Seeger, Terence Tao, Christoph Thiele, and James Wright.
\newblock A variation norm Carleson theorem.
\newblock {\em J. Eur. Math. Soc.} 14:421--464, 2012.

\bibitem{pisier:75}
Gilles Pisier.
\newblock Martingales with values in uniformly convex spaces.
\newblock {\em Israel. J. Math.}, 20(3-4):326--350, 1975.

\bibitem{pisier:86}
Gilles Pisier.
\newblock Probabilistic methods in the geometry of Banach spaces. 
\newblock In {\em Probability and analysis (Varenna, 1985)}, Springer, Berlin, 1986, 167--241.

\bibitem{pisier:11}
Gilles Pisier.
\newblock Martingales in {B}anach spaces (in connection with {T}ype and
  {C}otype). {C}ourse {IHP}, {F}eb.~2--8, 2011.
\newblock Manuscript, \url{http://www.math.jussieu.fr/~pisier/ihp-pisier.pdf}.

\bibitem{pisier:xu:88}
Gilles Pisier and Quanhua Xu.
\newblock The strong $p$-variation of martingales and orthogonal series.
\newblock {\em Prob.~Th.~Rel.~Fields} 77:497--514, 1988.

\bibitem{pourel:richards:89}
Marian~B. Pour-El and J.~Ian Richards.
\newblock {\em {Computability in analysis and physics}}.
\newblock Springer, Berlin, 1989.

\bibitem{nagy:et:al:10}
B\'ela Sz.-Nagy, Hari Bercovici, Ciprian Foias, L\'aszl\'o K\'erchy.
\newblock {\em Harmonic analysis of operators on Hilbert space}, second edition.
\newblock Springer, New York, 2010.

\bibitem{kohlenbach:schade:12}
Katharina Schade and Ulrich Kohlenbach.
\newblock Effective metastability for modified halpern iterations in CAT(0)
  spaces.
\newblock {\em Fixed Point Theory and Applications}, 2012:191, 19 pages, 2012.

\bibitem{tao:08}
Terence Tao.
\newblock Norm convergence of multiple ergodic averages for commuting
  transformations.
\newblock {\em Ergodic Theory Dynam. Systems}, 28(2):657--688, 2008.

\bibitem{vyugin:97}
V.~V. V'yugin.
\newblock Ergodic convergence in probability, and an ergodic theorem for
  individual random sequences.
\newblock {\em Teor. Veroyatnost. i Primenen.}, 42(1):35--50, 1997.

\bibitem{vyugin:98}
V.~V. V'yugin.
\newblock Ergodic theorems for individual random sequences.
\newblock {\em Theoret. Comput. Sci.}, 207(2):343--361, 1998.
\bibitem{walsh:12}
Miguel Walsh.
\newblock Norm convergence of nilpotent ergodic averages.
\newblock {\em Annals of Mathematics}, 175(3):1667--1688, 2012.
 
\end{thebibliography}

\end{document}